\documentclass[oneside,10pt]{article}          
\usepackage[b5paper]{geometry}	     
\usepackage{amsfonts,amsmath,latexsym,amssymb} 
\usepackage{theorem}                
\usepackage{enumerate}

\usepackage{mathrsfs,upref}        
\usepackage{mathptmx}		  
\usepackage{dea}	          

\newtheorem{theorem}{Theorem}           
\newtheorem{lemma}{Lemma}               
\newtheorem{corollary}{Corollary}

\theoremstyle{definition}

\newtheorem{example}{Example}

\newtheorem{remark}{Remark}

\newcommand{\mc}{\mathcal{C}}
\newcommand{\mr}{\mathbb{R}}
\newcommand{\bx}{\bar{x}}
\newcommand{\md}{\mathcal{D}}

\newcommand{\sL}{\mathcal{L}}

\begin{document}

\title[Short title]{Existence of Solutions to Nonlinear Legendre Boundary Value Problems}

\author{Benjamin Freedman and Jes\'{u}s Rodr\'{i}guez}

\address{Benjamin Freedman \\ 
Department of Mathematics \\
Box 8205, NCSU, Raleigh, NC 27695-8205\\
USA \\
\email{bnfreedm@ncsu.edu}}

\address{Jes\'{u}s Rodr\'{i}guez \\
Department of Mathematics \\
Box 8205, NCSU, Raleigh, NC 27695-8205 \\
USA \\
\email{rodrigu@ncsu.edu}}

\date{DD.MM.YYYY}                               



\begin{abstract}
In this paper, we consider nonlinearly perturbed Legendre differential equations subject to the usual boundary conditions. For such problems we establish sufficient conditions for the existence of solutions and in some cases we provide a qualitative description of solutions depending on a parameter. The results presented depend on the size and limiting behavior of the nonlinearities.
\end{abstract}

\maketitle

\section{Introduction}
In this paper, we discuss the solvability of boundary value problems which arise as nonlinear perturbations of the classical Legendre differential equation subject to the standard boundary conditions. The framework we present enables us to establish the existence of solutions to boundary value problems under a variety of conditions. Each approach takes advantage of the general linear Sturm-Liouville theory, in particular existing knowledge regarding the spectrum of the Legendre Sturm-Liouville operator. In section 2.1, we provide a general framework that enables us to discuss the nonlinear boundary value problem as an operator equation of the form $Lx=F(x)$ and we establish conditions for the existence of solutions in the case where the linear part $L$ is invertible. In sections 2.2 and 2.3, we don't make this invertibility assumption and results we obtain are  based on the projection scheme commonly referred to as the Lyapunov-Schmidt procedure. In section 2.2, we use fixed-point theorems to provide sufficient conditions for the solvability of the boundary value problem that depend on the limiting behavior of the nonlinearity. In section 2.3, the same projection scheme along with the implicit function theorem is used to establish the existence and qualitative properties of solutions to weakly nonlinear problems.

\indent Approaches similar to the ones appearing in this paper have been used in a variety of settings in the study of nonlinear boundary value problems. For the use of arguments similar to those in section 2.1 in the continuous and discrete cases, the reader is referred to \cite{bou}, \cite{freed}, \cite{lazlea}, \cite{sl05}, \cite{ja1}, and \cite{jsuar1}. For the general theory of projection methods in nonlinear boundary value problems we suggest \cite{sweet}. For the use of projection methods similar to those in subsections 2.2 and 2.3, see \cite{drab}, \cite{eth1}, \cite{marrod2}, \cite{marrod}, \cite{pad}, \cite{jfrod}, and \cite{ura}. For results involving topological degree theory arguments in the analysis of discrete boundary value problems, the reader may consult \cite{chendu} and \cite{yujin}. \\
\indent The classical Legendre eigenvalue-eigenfunction problems consists of finding the scalars $\mu$ and functions $x: (-1,1) \to \mr$ such that 
\begin{align}
[(1-t^2)x'(t)]'+\mu x(t)=0 \nonumber
\end{align}
for all $t \in (-1,1)$ where
\begin{align}
\lim_{t \to -1^+} x(t) ,\indent 
\lim_{t \to 1^-} x(t)  \nonumber \\
  \nonumber \\
\lim_{t \to -1^+} x'(t),\indent  \nonumber
\lim_{t \to 1^-} x'(t). \nonumber 
\end{align}
all exist and are finite. It is well-known that nontrivial solutions of this problem exist if and only if $\mu=k(k+1)$ where $k$ is a nonnegative integer. If $\mu=k(k+1)$, the only solutions are the constant multiples of the $k^{th}$ Legendre polynomial. In this paper, we consider a nonlinear perturbation of the differential equation subject to the same boundary conditions. That is, the existence of finite limits of $x(t)$ and $x'(t)$ at $1$ and $-1$.

\section{Differential Equations}

\subsection{The Case of Invertible $L$}
Even though in this paper we are mainly interested in the cases where the parameter $\mu$ in the equation below is an eigenvalue of the associated linear Legendre equation, we devote this first section to the case where $\mu \neq k(k+1)$ for any nonnegative integer $k$. We consider boundary value problems on $(-1,1)$ of the form,
\begin{align}
[(1-t^2)x'(t)]'+\mu x(t)=f(x(t)) \label{lbvp}
\end{align}
subject to the condition that the following limits exist and are finite
\begin{align}
\lim_{t \to -1^+} x(t) \nonumber , \indent
\lim_{t \to 1^-} x(t)  \\
\label{bc} \\
\lim_{t \to -1^+} x'(t) \nonumber \indent
\lim_{t \to 1^-} x'(t). \nonumber 
\end{align}
Throughout this paper, we assume that $f: \mr \to \mr$ is continuous. Let $\mathcal{L}^2$ denote the space of functions $\mathcal{L}^2=(\mathcal{L}^2[-1,1], \| \cdot \|_2)$, $X$ be defined as the subspace of functions in $\mathcal{L}^2$ where the limits appearing in \eqref{bc} exist and are finite and 
\begin{align}
D(L)=\{ x \in X: x' \text{ is absolutely continuous and } x'' \in \mathcal{L}^2\}. \nonumber
\end{align} 
In this section, we assume that $f: \mr \to \mr$ is Lipschitz.

This implies that $f \circ x \in L^2$ for all $x \in L^2$. We seek conditions under which we can guarantee the existence of a solution to the boundary value problem \eqref{lbvp}-\eqref{bc}. \\

\indent We now present some basic results regarding a closely related linear boundary value problem. If $\mu \neq k(k+1)$ for all $k$, the equation
\begin{align}
[(1-t^2)x'(t)]'+\mu x(t)&=h(t) \nonumber 
\end{align}
has exactly one solution satisfying the condition that the following limits exist and are finite
\begin{align}
\lim_{t \to -1^+} x(t) \nonumber , \indent
\lim_{t \to 1^-} x(t)  \\
 \nonumber \\
\lim_{t \to -1^+} x'(t) \nonumber \indent
\lim_{t \to 1^-} x'(t). \nonumber 
\end{align}
Define the map $L: D(L) \to \mathcal{L}^2$ by
\begin{align}
[Lx](t)=[(1-t^2)x'(t)]'+\mu x(t). \nonumber
\end{align} 
Clearly, if $\mu \neq k(k+1)$ for all $k$ then $L$ is a bijection from $D(L)$ onto $L^2$.

Let $P_k$ denote the $k^{th}$-degree Legendre polynomial and $p(t)=(1-t^2)$. From general Sturm-Liouville theory, the equation $(px')'+\lambda x=0$, subject to the condition that the limits in \eqref{bc} exist and are finite, has countably many simple eigenvalues $\lambda_k=k(k+1)$ with corresponding eigenfunctions $P_k$ for $k \geq 0$. It is also well-known that $L$ is self-adjoint and that the graph of $L$ is closed.
Further, the unique solution $x_h \in D(L)$ to $Lx=h$ guaranteed above can be represented by the eigenfunction expansion,
\begin{align}
x_h= \sum_{k=0}^{\infty} \frac{(k+\frac{1}{2}) \langle h, P_k \rangle}{[\mu-k(k+1)]} P_k \nonumber
\end{align}
where $\langle \cdot, \cdot \rangle$ denotes the standard $\mathcal{L}^2$ inner product. From this is follows that $L^{-1}$ is continuous and that
\begin{align}
\big\|L^{-1}\big\| \leq \Bigg( \sum_{k=0}^{\infty} \Bigg| \frac{1}{(\mu-k(k+1))^2(k+\frac{1}{2})} \Bigg| \Bigg)^{1/2}. \nonumber
\end{align}

This information, as well as more on the general theory of Legendre polynomials and the Legendre differential equation can be found in  \cite{can}.

The following corollary establishes the continuity of $L^{-1}$ by giving a bound on its operator norm that will be useful later.
Before presenting the next corollary, we first must introduce some notation. Let $\mc$ denote the space of continuous functions on $[-1,1]$, and $\| \cdot \|_{\infty}$ denote the supremum norm. That is, for a continuous function $x$ on $[-1,1]$,
\begin{align}
\|x\|_{\infty}=\sup_{t \in [-1,1]} |x(t)|. \nonumber
\end{align}
\begin{corollary} \label{cor1}
There exists $K>0$ such that for all $h \in Im(L) \subset \mathcal{L}^2$, the unique solution $x_h$ to the equation $Lx=h$ satisfies
\begin{align}
\| x_h \|_{\infty} \leq K \|h\| \nonumber
\end{align}
and 
\begin{align}
\|x_h' \|_{\infty} \leq K \|h\|. \nonumber
\end{align}
\end{corollary}
\begin{proof}
Define the map $\hat{L}: \hat{D}(L) \to Im(L)$ by 
\begin{align}
[\hat{L}x](t)=[(1-t^2)x'(t)]'+\mu x(t) \nonumber
\end{align}
where $\hat{D}(L)$ consists of the same set of functions as $D(L)$ but is endowed with the norm $\| \cdot \|_{H^2}$ given by
\begin{align}
\| z \|_{H^2}= \|z\|_{\infty}+\|z'\|_{\infty}+ \|z'' \|. \nonumber
\end{align}
Note that the map $\hat{L}$ is a continuous, linear bijection onto $\mathcal{L}^2$, and that $\hat{D}(L)$ and $Im(L)$ are Banach spaces. Therefore, by a consequence of the open mapping theorem, $\hat{L}^{-1}$ is continuous. This means there exists a $K>0$ such for any $h \in \mathcal{L}^2$ the unique solution $x_h$ to $Lx=h$ satisfies:
\begin{align}
K\|h\| \geq \|x_h\|_{H^2} \geq \|x_h\|_{\infty} \nonumber
\end{align}
and
\begin{align}
K\|h\| \geq \|x_h\|_{H^2} \geq \|x'_h\|_{\infty} \nonumber
\end{align}
as required. \qed \medskip
\end{proof}

\begin{lemma} \label{lemma1}
The map $L^{-1}: Im(L) \to \mathcal{L}^2$ is compact.
\end{lemma}
\begin{proof}
Consider the map $\tilde{L}: \tilde{D}(L) \to Im(L)$ defined by
\begin{align}
[\tilde{L}x](t)=[(1-t^2)x'(t)]'+\mu x(t) \nonumber
\end{align}
where $\tilde{D}(L)$ consists of the same set of functions as $D(L)$ but endowed with the norm $\| \cdot \|_{\infty}$. Note that $\tilde{L}$ is invertible due to the fact that $L$ is invertible. We wish to show that $\tilde{L}$ is compact using the Arzela-Ascoli theorem.
Let $M>0$ and define $S$ to be the set $S=\{ z \in Im(L) : \|z\| \leq M \}$. Let $h \in S$ and observe that 
\begin{align}
\|\tilde{L}^{-1} h \|_{\infty} &\leq K \|h\| \nonumber \\
&\leq K M. \nonumber
\end{align}
Therefore, the $\tilde{L}^{-1}(S)$ is a uniformly bounded set of functions in $\mc$. We now wish to show that this set is equicontinuous. \\ \\
Let $h \in S$ and let $\varepsilon>0$. By the previous corollary along with the mean value theorem, for any $h \in \mathcal{L}^2$ $\tilde{L}^{-1}h$ is Lipschitz on $S$ with constant $KM$. Let $\delta=\varepsilon/KM$ and $|t_1-t_2|<\delta$.
Then we have that 
\begin{align}
|\tilde{L}^{-1}h(t_1)-\tilde{L}^{-1}h(t_2)| &\leq KM |t_1-t_2| \nonumber \\
&<\varepsilon. \nonumber
\end{align}
Therefore,  $\tilde{L}^{-1}(S)$ is an equicontinuous set of functions in $\mc$. By the Arzel\'{a}-Ascoli theorem, $\tilde{L}^{-1}: Im(L) \to \bar{D}(L)$ is compact. Therefore, it follows that $L^{-1}: Im(L) \to D(L)$ is a compact operator. \qed \medskip 

\end{proof}

We now discuss the issue of whether we can guarantee a solution  to the nonlinear boundary value problem
\begin{align}
[(1-t^2)x'(t)]'+\mu x(t)=f(x(t)) \nonumber
\end{align}
where $x$ satisfies the condition that limits in \eqref{bc} exist and are finite. Define $F: \mathcal{L}^2 \to \mathcal{L}^2$ by
\begin{align}
F(x)=f \circ x. \nonumber
\end{align}
It is evident that the boundary value problem \eqref{lbvp}-\eqref{bc} is equivalent to the operator equation $Lx=F(x)$.

\begin{theorem} \label{theo1}
Suppose that $f: \mr \to \mr$ is Lipschitz and that $\mu \neq k(k+1)$ for all nonnegative integers $k$. Then if 
\begin{align}
\lim_{|s| \to \infty} \frac{|f(s)|}{|s|}=0 \nonumber
\end{align}
there exists a solution to the boundary value problem 
\begin{align}
[(1-t^2)x'(t)]'+\mu x(t)=f(x(t)) \nonumber
\end{align}
subject to the condition that the limits in \eqref{bc} exist and are finite.
\end{theorem}

The proof of this theorem is a standard application of Schauder's fixed point theorem applied to the operator $L^{-1} F$. We omit the details. Note that results in this subsection depended heavily on $L$ having an inverse, which is only the case if we assume $\mu \neq k(k+1)$ for any $k \in \mathbb{N}$. The following subsections analyze situations where $L$ is not invertible.

\subsection{The Case of Non-Invertible $L$}
We will now assume that $\mu=k(k+1)$ for some $k \in \{0,1,2, \dots\}$. As a consequence of the general Sturm-Liouville theory outlined in the previous section, $\mu=k(k+1)$ implies that  the kernel of $L$ is one-dimensional and spanned by $P_k$. Further, as stated in \cite{hol}, we have that $h \in im(L)$ if and only if 
\begin{align}
\langle h, P_k \rangle=0. \nonumber
\end{align}
Therefore, it follows that $Im(L)=[\ker(L)]^{\perp}$.
In this section we will assume that $\lim_{s \to \infty} f(s)$ and $\lim_{s \to -\infty} f(s)$ exist and are finite. We denote these values by
\begin{align}
f(\infty) \equiv \lim_{s \to \infty} f(s) \indent \text{  and  } \indent
f(-\infty) \equiv \lim_{s \to -\infty} f(s). \nonumber 
\end{align}

We employ the Lyapunov-Schmidt procedure. For the readers convenience, we now outline the basic elements of this process. For more details, the reader may consult \cite{hale1}. \\
\indent First define $U:\mathcal{L}^2 \to \mathcal{L}^2$ by 
\begin{align}
[Ux](t)=\bigg( k+\frac{1}{2} \bigg) \langle x, P_k \rangle P_k(t). \nonumber
\end{align}
Note that $U$ is a projection onto $\ker(L)=span\{P_k\}$. Define the projection $E: \mathcal{L}^2 \to \mathcal{L}^2$ onto $[\ker(L)]^{\perp} =Im(L)$ by $E=I-U$.  Note that the map $L$ restricted to $D(L) \cap Im(L)$ is a bijection onto $Im(L)=Im(E)$. Therefore, it follows that there exists a linear map $M: Im(E) \to D(L) \cap Im(L)$ satisfying $LMh=h$ for all $h \in Im(L)$ and $MLx=Ex=(I-U)x$  for all $x \in D(L)$. In fact, we can represent this map $M$ with the eigenfunction expansion,
\begin{align}
[Mh](t)=\sum_{l \neq k} \frac{ (l+\frac{1}{2} )\langle h, P_l \rangle}{[\mu-l(l+1)]} P_l(t). \nonumber
\end{align}
Note that $M: Im(L) \to Im(L) \cap D(L)$ is a compact operator as a consequence of the argument appearing in lemma \ref{lemma1} along with the fact that $Im(L)$ is a closed subspace of $\mathcal{L}^2$.
Using these projections, we analyze the operator equation $Lx=F(x)$ in the following way:
  \begin{align}
Lx=F(x) &\Longleftrightarrow 
  \begin{cases}
      E(Lx-F(x))=0 \nonumber \\
     \text{\indent and} \nonumber \\
      (I-E)(Lx-F(x))=0 \nonumber 
    \end{cases} \nonumber \\ \nonumber \\
&\Longleftrightarrow 
\begin{cases}
      (I-U)x-MEF(x)=0 \nonumber \\
     \text{\indent and} \nonumber \\
      F(x) \in Im(L) \nonumber
    \end{cases}   \nonumber \\ \nonumber \\
    &\Longleftrightarrow 
\begin{cases}
      x=Ux+MEF(x) \nonumber \\
    \text{\indent and}  \nonumber \\
      \int_{-1}^1 f(x(t)) P_k(t) dt =0 \nonumber
    \end{cases}    \nonumber \\ \nonumber \\
    &\Longleftrightarrow 
 \begin{cases}
      x=\alpha P_k+w(x) \nonumber \\
   \text{\indent and}   \nonumber \\
      \int_{-1}^1 f[\alpha P_k(t)+w(x(t))] P_k(t) dt =0 \nonumber
    \end{cases}    \nonumber
\end{align} 
where $w(x)=MEF(x)$. \\

Define the constants $J_1$ and $J_2$ as follows:
\begin{align}
J_1&=f(\infty) \int_{\{t:P_k(t)>0\}} P_k(t) dt+ f(-\infty) \int_{\{t:P_k(t)<0\}} P_k(t) dt \nonumber \\
J_2&=f(\infty) \int_{\{t:P_k(t)<0\}} P_k(t) dt+ f(-\infty) \int_{\{t:P_k(t)>0\}} P_k(t) dt. \nonumber
\end{align}

Note that if $k=0$, then $J_1=g(\infty)$ and $J_2=g(-\infty)$. If $k \geq 1$, then 
\begin{align}
J_1&=\bigg( \int_{\{t:P_k(t)>0\}} P_k(t) dt \bigg) [f(\infty)-f(-\infty)] \nonumber \\
J_2&=\bigg( \int_{\{t:P_k(t)>0\}} P_k(t) dt \bigg) [f(-\infty)-f(\infty)]. \nonumber
\end{align}

\begin{theorem} \label{theo2}
Suppose that $f: \mr \to \mr$ is continuous and that $f(-\infty)$ and $f(\infty)$ exist and are finite.
Then we can guarantee a solution to the boundary value problem \eqref{lbvp}-\eqref{bc} in either of the following cases:
\begin{enumerate}[i)]
\item  If $k=0$ and $f(-\infty) f(\infty)<0$
\item If $k \geq 1$ and $f(-\infty) \neq f(\infty).$
\end{enumerate}  
\end{theorem}
\begin{proof}
We begin by noting that,
\begin{multline}
\indent \int_{-1}^1 f[\alpha P_k(t)+w(x(t))] P_k(t) dt
\\ =  \int_{\{t:P_k(t)<0\}} f[\alpha P_k(t)+w(x(t))] P_k(t) dt+ \int_{\{t:P_k(t)>0\}} f[\alpha P_k(t)+w(x(t))] P_k(t) dt. \nonumber
\end{multline}
Since $w$ is bounded, we have by the Lebesgue Dominated Convergence Theorem that
\begin{align}
\lim_{\alpha \to \infty} \int_{-1}^1 f[\alpha P_k(t)+w(x(t))] P_k(t) dt=f(\infty) \int_{\{t:P_k(t)>0\}} P_k(t) dt+ f(-\infty) \int_{\{t:P_k(t)<0\}} P_k(t) dt=J_1 \nonumber 
\end{align}
and 
\begin{align}
\lim_{\alpha \to -\infty} \int_{-1}^1 f[\alpha P_k(t)+w(x(t))] P_k(t) dt=f(\infty) \int_{\{t:P_k(t)<0\}} P_k(t) dt+ f(-\infty) \int_{\{t:P_k(t)>0\}} P_k(t) dt=J_2. \nonumber 
\end{align}
Condition $iii)$ implies that $J_1J_2<0$, and we proceed by supposing without loss of generality that $J_2<0<J_1$. 
\end{proof}
Therefore there exists $\alpha_0 >0$ such that for $\alpha \geq \alpha_0$,
\begin{align}
\int_{-1}^1 f[\alpha P_k(t)+w(x(t))] P_k(t) dt>0  \label{geq}
\end{align}
and for $\alpha \leq -\alpha_0$,
\begin{align}
\int_{-1}^1 f[\alpha P_k(t)+w(x(t))] P_k(t) dt<0. \label{leq}
\end{align}
Note that $M$ is a compact linear map from $Im(L)$ onto $D(L) \cap Im(L)$ and $E$ is a projection so $w$ is a nonlinear compact mapping. \\ \\
Define $H_1: \mathcal{L}^2 \times \mr \to \mathcal{L}^2$ by 
\begin{align}
H_1(x, \alpha)= \alpha P_k+w(x) \nonumber
\end{align}
and $H_2: \mathcal{L}^2 \times \mr \to \mr$ by 
\begin{align}
H_2(x,\alpha)= \alpha-\int_{-1}^1 f[\alpha P_k(t)+w(x(t))] P_k(t) dt. \nonumber
\end{align}
Let $H: \mathcal{L}^2 \times \mr \to \mathcal{L}^2 \times \mr$ be defined by 
\begin{align}
H(x,\alpha)=(H_1(x,\alpha), H_2(x,\alpha)). \nonumber
\end{align}
Guaranteeing a fixed point for $H$ is equivalent to guaranteeing a solution to \eqref{lbvp}-\eqref{bc}. We endow the space $\mathcal{L}^2 \times \mr$ with the norm
\begin{align}
\|(x,\alpha)\|=\max\{ \|x\|, | \alpha| \}. \nonumber
\end{align}
Define 
\begin{align}
r&=\sup_{t \in \mr} |f(t)|. \nonumber
\end{align}
The existence of $r$ is guaranteed by the continuity of $f: \mr \to \mr$ along with the fact that $f(\infty)$ and $f(-\infty)$ exist and are finite. Choose $\alpha_0>r$ so that $\eqref{geq}$ and $\eqref{leq}$ are satisfied and let $\delta=\alpha_0+r$. As stated in \cite{can}, $|P_k(t)| \leq 1$ for all $t \in [-1,1]$. We know that $f$ and $ME$ are bounded, so there exists $b_1>0$ such that for any $x \in \mathcal{L}^2$, $\alpha \in \mr$,
\begin{align}
\|H_1(x, \alpha)\| \leq b_1. \nonumber
\end{align}
Let $\mathcal{B}$ be the set
\begin{align}
\mathcal{B}=\{ (x,\alpha) \in \mathcal{L}^2 \times \mr: \|x\| \leq b_1, |\alpha| \leq \delta\}. \nonumber
\end{align}
Clearly $\|H_1(x, \alpha) \| \leq b_1$ for all $(x, \alpha) \in \mathcal{B}$ by construction, so it suffices to show that $\|H_2(x, \alpha)\| \leq \delta$ for all $(x,\alpha) \in \mathcal{B}$ in order to show that $H(\mathcal{B}) \subset \mathcal{B}$. \\ \\
Suppose that $\alpha \in [\alpha_0,\delta]$. Then 
\begin{align}
\int_{-1}^1 f[\alpha P_k(t)+w(x(t))] P_k(t) dt>0 \nonumber
\end{align}
and therefore $H_2(x,\alpha)<\alpha \leq \delta$. Further, since $\bigg| \int_{-1}^1 f[\alpha P_k(t)+w(x(t))] P_k(t) dt \bigg| \leq r$ it follows that 
\begin{align}
\alpha-\int_{-1}^1 f[\alpha P_k(t)+w(x(t))] P_k(t) dt \geq \alpha_0-r \geq 0. \nonumber
\end{align}
Therefore if $\alpha \in [\alpha_0,\delta]$ then $|H_2(x, \alpha)| \in [0,\delta]$.
Suppose that $\alpha \in [0, \alpha_0)$. Then 
\begin{align}
|H_2(x,\alpha)|&= \bigg| \alpha-\int_{-1}^1 f[\alpha P_k(t)+w(x(t))] P_k(t) dt \bigg| \nonumber \\
&\leq \alpha_0+r \nonumber \\
&=\delta. \nonumber
\end{align}
Therefore, if $(x, \alpha) \in \mathcal{B}$ and $\alpha \in [0,\delta]$ then $|H_2(x,\alpha)| \leq \delta$. \\
\indent A symmetric argument can be used to show that if $(x, \alpha) \in \mathcal{B}$ and $\alpha \in [-\delta, 0]$ then $|H_2(x,\alpha)| \leq \delta$. Therefore, $H(\mathcal{B}) \subset \mathcal{B}$. Since $H: \mathcal{L}^2 \times \mr \to \mathcal{L}^2 \times \mr$ is compact (following from the compactness of $M$) and $\mathcal{B}$ is closed, bounded, and convex it follows that $H$ is guaranteed a fixed point by Schauder's Fixed Point Theorem. \qed \medskip

\subsection{The Case of Weak Nonlinearities}
In this subsection, assume that our nonlinearity is of the form $\varepsilon f(x(t))$ where $\varepsilon$ is a real parameter and $f: \mr  \to \mr$ is continuously differentiable. That is, we now examine boundary value problems of the form
\begin{align}
[(1-t^2)x'(t)]'+\mu x(t)=\varepsilon f(x(t)) \nonumber 
\end{align}
 subject to the condition that the limits appearing in \eqref{bc} exist and are finite. Due to the fact that we will impose differentiability conditions on the function-valued operator representing our nonlinearity, we consider operators defined on the space of continuous functions. Again let $\mc$ denote the space of  continuous functions on $[-1,1]$ endowed with the supremum norm and let 
\begin{align}
\mathcal{D}=\big( C^2[-1,1] , \| \cdot \|_{\infty} \big) \subset \mc \nonumber
\end{align} 
where $C^2[-1,1]$ denotes the set of twice continuously differentiable functions on $[-1,1]$.
In this section, denote $\sL: \mathcal{D} \to \mc$ by
\begin{align}
[\sL x](t)=[(1-t^2)x'(t)]'+\mu x(t) \nonumber
\end{align}
and $F: \mc \times \mr \to \mc$ by
\begin{align}
[F (x,\varepsilon)](t)=\varepsilon f(x(t)). \nonumber
\end{align}
Suppose again that $\mu=k(k+1)$.

In this section, for $x \in \mc$ and $l \in \mathbb{N}$ we denote
\begin{align}
x_l=\Bigg[\bigg( l+\frac{1}{2} \bigg) \int_{-1}^1 P_l(t) x(t) dt \Bigg]. \nonumber
\end{align}
Define the projections $U: \mc \to \mc$ by 
\begin{align}
[Ux](t)=x_kP_k(t) \nonumber
\end{align} 
and $E: \mc \to \mc$ by $E=I-U$. Note that the map $\sL$ restricted to $\mathcal{D} \cap Im(L)$ is a bijection onto $Im(L)=Im(E)$. Therefore, it follows that there exists a linear map $M: Im(E) \to \mathcal{D} \cap Im(L)$ satisfying
\begin{align}
\sL Mh=h \nonumber
\end{align}
for all $h \in Im(L)$ and 
\begin{align}
M\sL x=Ex=(I-U)x \nonumber
\end{align}
for all $x \in \mathcal{D}$.
Note that $M$ is simply
\begin{align}
\bigg[ \sL|_{\mathcal{D} \cap Im(L)} \bigg]^{-1} \nonumber
\end{align}
 and observe further that $M$ is continuous.
We note that solving 
\begin{align}
\sL x=F(x, \varepsilon) \nonumber
\end{align}
is equivalent to solving the system
\begin{align}
 \begin{cases}
      (I-U)x-MEF(x, \varepsilon)=0 \nonumber \\
     \text{ \indent and} \nonumber \\
      U(f \circ x)=0. \nonumber
    \end{cases}  \nonumber 
    \end{align}
    Define the map $G: \mathcal{D} \times \mr \to Im(L) \times \ker(L)$ by
\begin{align}
G(x, \varepsilon)=\bigg[
\begin{array}{c}
(I-U)x-MEF(x, \varepsilon) \nonumber \\
U(f \circ x) \nonumber
\end{array}
\bigg].
\end{align}
It is well known that $F$ is continuously differentiable with respect to $x$ and for any $x \in \mc$, $\varepsilon \in \mr$,
\begin{align}
\bigg(\frac{\partial F}{\partial x} (x, \varepsilon) h\bigg)(t)=\varepsilon f'(x(t)) h(t). \nonumber
\end{align}

From that it follows that
\begin{align}
\frac{\partial G}{\partial x}(x, \varepsilon) \nonumber
\end{align}
exists for all $(x, \varepsilon) \in \mc \times \mr$ and is given by
\begin{align}
\frac{\partial G}{\partial x}(x, \varepsilon)w=\Bigg[
\begin{array}{c}
[(I-U)-\varepsilon ME(f' \circ x)] w \nonumber \\
U(f' \circ x)w  \nonumber
\end{array}
\Bigg].
\end{align}
Let $\bx=\alpha P_k$ for $\alpha \in \mr$. For $(\bx,0)$ and $w \in \mc$:
\begin{align}
\frac{\partial G}{\partial x}(\bx, 0)w=\bigg[
\begin{array}{c}
(I-U)w \nonumber \\
U(f' \circ \bx)w  \nonumber
\end{array}
\bigg].
\end{align}
Since $F \in C^1$ and $M$ is continuous, it follows that $G \in C^1$.
For $w \in \mc$, we can decompose $w$ as $w=u+v$ where 
\begin{align}
u&=w_kP_k \nonumber \\
v&=w-w_kP_k. \nonumber
\end{align} . With this in mind, we can rewrite the previous expression as
\begin{align}
\frac{\partial G}{\partial x}(\bx, 0)(u+v)=\bigg[
\begin{array}{c}
v \nonumber \\
U(f' \circ \bx)(u+v)  \nonumber
\end{array}
\bigg].
\end{align}
Define the maps $H_1: \ker(L) \to \mr$ by
\begin{align}
H_1(u)=\int_{-1}^1 P_k(t) f(u(t)) dt, \nonumber
\end{align}
$H_2:\mr \to \ker(L)$ by $H_2(\alpha)=\alpha P_k$ and finally $H: \mr \to \mr$ by $H=H_1 \circ H_2$. That is,
\begin{align}
H(\alpha)=\int_{-1}^1 P_k(t) f(\alpha P_k(t)) dt. \nonumber
\end{align}
Therefore for any number in $\mr$, $H': \ker(L) \to \mr$ exists and for $\beta \in \mr$,
\begin{align}
[H'(\alpha )](\beta)=\int_{-1}^1 P_k(t) [f'(\alpha P_k(t))](\beta P_k(t)) dt. \nonumber
\end{align}
We are now ready to give conditions for the solvability of our boundary value problems examined this section.

\begin{theorem} \label{theo3}

Suppose that there exists $\alpha_0 \in \mr$ such that $H(\alpha_0 )=0$ and $H'(\alpha_0 ) \neq 0$. 
Then there exists and open neighborhood $I \subset \mr$ of $0$ such that for any $\varepsilon \in I$ there exists a solution to
\begin{align}
[(1-t^2)x'(t)]'+\mu x(t)=\varepsilon f(x(t))  \nonumber
\end{align}
satisfying the condition that the limits appearing in \eqref{bc} exist and are finite.
\end{theorem}
\begin{proof}
Recall that $G \in C^1$ and let $\bx=\alpha_0 P_k$. Then $(I-U)\bx-MEF(\bx, 0)=0$ and 
\begin{align}
UF(\bx)&=\int_{-1}^1 P_k(t) f(\alpha_0P_k(t)) dt \nonumber \\
&=H(\alpha_0P_k(t) ) \nonumber \\
&=0. \nonumber
\end{align}

Therefore $G(\bx, 0)=0$. We now wish to show that $\frac{\partial G}{\partial x}( \bx, 0)$ is a bijection from $\mc$ onto $Im(L) \times \ker(L)$. Since $\frac{\partial G}{\partial x}( \bx, 0)$ is linear, in order to show this map is injective it suffices to show that it has a trivial kernel. Suppose that $\frac{\partial G}{\partial x}( \bx, 0)(u+v)=0$. Then 
\begin{align}
0&=v \nonumber
\end{align}
and so 
\begin{align}
0=U(f' \circ \bx)u=\bigg[ \int_{-1}^1 P_k(t) [f'(\alpha_0 P_k(t))]u(t) dt \bigg] \nonumber
\end{align}
implying that $u=0$ due to our assumption that $H'(\alpha_0 ) \neq 0$. Note that since $H'(\alpha_0 )$ is a nonzero linear map from $\mr \to \mr$, then it is a bijection from $\mr$ onto $\mr$. This implies that the map $U(f' \circ \bx)$ restricted to $\ker(L)$ is a bijection onto $\ker(L)$. Given $h_1 \in Im(L)$ and $h_2 \in \ker(L)$, we have that 
\begin{align}
\frac{\partial G}{\partial x}( \bx, 0) (h_1+\hat{h}_2)=(h_1,h_2) \nonumber
\end{align}
where $\hat{h}_2$ is the unique element of $\ker(L)$ that maps to $h_2$ under $U(f' \circ \bx)$.
So $\frac{\partial G}{\partial x}( \bx, 0)$ is surjective and therefore a bijection from $\mc$ onto $Im(L) \times \ker(L)$.
 By the implicit function theorem \cite{lang}, there exists a neighborhood $V_0 \subset \mr$ of $0$ on which there exists a continuous function $\phi:V_0 \to \md$ satisfying
\begin{align}
G(\phi(\varepsilon), \varepsilon)=0 \nonumber
\end{align}
for all $\varepsilon \in V_0$. Denoting $\phi(\varepsilon)=x_\varepsilon$ we have that 
\begin{align}
0&=G(\phi(\varepsilon), \varepsilon) \nonumber \\
&=G(x_\varepsilon, \epsilon) \nonumber \\
&=\mathcal{L}x_\varepsilon-F(x_\varepsilon,\varepsilon). \nonumber
\end{align}
In other words for any $\varepsilon \in V_0$ we can guarantee a solution to 
\begin{align}
[(1-t^2)x'(t)]'+\mu x(t)=\varepsilon f(x(t)) \nonumber
\end{align}
satisfying the condition that the limits in \eqref{bc} exist and are finite. \qed \medskip
\end{proof}

\begin{remark}
Let $x_\varepsilon$ denote the solution in $\md$ guaranteed by the implicit function theorem to 
\begin{align}
[(1-t^2)x'(t)]'+\mu x(t)=\varepsilon f(x(t)).  \nonumber
\end{align}
Note that 
\begin{align}
\lim_{\varepsilon \to 0} x_\varepsilon=\bx \nonumber 
\end{align}
where this limit is in the sense of uniform convergence. That is, solutions guaranteed by the above theorem are ones that emanate from a certain solution to the linear homogeneous problem.
\end{remark}

\begin{example}
Consider the boundary value problem
\begin{align}
[(1-t^2)x'(t)]'=\varepsilon f(x(t)) \nonumber
\end{align}
on $(-1,1)$ subject to the condition that the limits in \eqref{bc} exist and are finite.

Suppose that there exists a number $\alpha_0$ such that $f(\alpha_0)=0$ and $f'(\alpha_0) \neq 0$. Then since the constant Legendre polynomial is $P_0(t)=1$, for $\alpha \in \mr$,
\begin{align}
\int_{-1}^1 P_0(t) f(\alpha P_0(t)) dt \nonumber &=\int_{-1}^1 f(\alpha) dt
\end{align}
so then $\int_{-1}^1 P_0(t) f(\alpha_0 P_0(t)) dt=0$. However, provided $\beta \neq 0$,
\begin{align}
\int_{-1}^1 P_0(t) [f'(\alpha_0 P_0(t))](\beta P_0(t)) dt &=\beta\int_{-1}^1 f'(\alpha_0 ) dt \nonumber
\end{align}
so then $\int_{-1}^1 P_0(t) [f'(\alpha_0 P_0(t))](\beta P_0(t)) dt \neq 0$. 
\end{example}

\begin{example}
Consider the boundary value problem
\begin{align}
[(1-t^2)x'(t)]'+2x(t)=\varepsilon f(x(t)) \nonumber
\end{align}
subject to the condition that the limits in \eqref{bc} exist and are finite
The constant Legendre polynomial is $P_1(t)=t$, so the condition in theorem \ref{theo3} is satisfied provided there exists a number $\alpha_0$ satisfying
\begin{align}
\int_{-1}^1 t f(\alpha_0 t) dt \nonumber
\end{align}
and $f(\alpha_0) \neq f(-\alpha_0)$.

\end{example}


\begin{thebibliography}{20}

  \bibitem{ahnie}
        {\sc  B. AHMAD, J.J. NIETO},
        {\it Existence of solutions for nonlocal boundary value problems of higher-order nonlinear fractional differential equations}, Abstract and Applied Analysis \textbf{2009} (2009).
 
  \bibitem{bou}
        {\sc  A. BOUCHERIF},
        {\it Second-order boundary value problems with integral boundary conditions}, Nonlinear Anal., \textbf{70}, (2009) 364--371.    


  \bibitem{can}
        {\sc P. CANUTO, M.Y. HUSSAINI, A. QUATERONI, T.A. ZANG},
        {\it Spectral Methods in Fluid Dynamics}, Springer--Verlag, Berlin Heidelberg 1988.  
      
 \bibitem{chendu}
        {\sc  X. CHEN, Z. DU},
        {\it Existence of positive periodic solutions for a neutral delay predator-prey model with Hassell-Varley type functional response and impulse}, Qual. Theory Dyn. Syst., (2018), 67--80.
        
  \bibitem{hale1}
        {\sc S. CHOW, J. K. HALE},
        {\it Methods of Bifurcation Theory}, Spring, Berlin, 1982.          
        
 \bibitem{drab}
        {\sc  P. DR\'{A}BEK},
        {\it Landesman-Lazer type condition and nonlinearities with linear growth}, Czechoslovak Math J., \textbf{40} (1990) 70--86.  

  \bibitem{dramil}
        {\sc P. DR\'{A}BEK, J. MILOTA},
        {\it Methods of Nonlinear Analysis: Applications to Differential Equations}, Birkha\"{u}ser Verlag AG, Basel, 2007.  
        
  \bibitem{yujin}
        {\sc  Z.I. DU, J. YIN},
        {\it A second order differential equation with generalized Sturm-Liouville integral conditions at resonance}, Filomat, \textbf{28} 7 (2014) 1437--1444.      

 \bibitem{eth1}
        {\sc D.L. ETHERIDGE, J. RODR\'{I}GUEZ},
        {\it Scalar discrete nonlinear two-point boundary value problems}, Journal of Difference Equations and Applications, \textbf{4}, 2 (1998), 127-144.       

\bibitem{freed}
        {\sc B. FREEDMAN, J. RODR\'{I}GUEZ},
        {\it On the Solvability of Nonlinear Differential Equations Subject to Generalized Boundary Conditions}, Differential Equations and Applications, \textbf{10}, 3 (2018), 317-327.

  \bibitem{hol}
        {\sc S. HOLLAND},
        {\it Applied Analysis by the Hilbert Space Method: An Introduction with Applications to the Wave, Heat, and Schr\"{o}dinger Equations}, Dover Publications, Dover Ed edition, 2007. 
        
\bibitem{lang}
        {\sc  S. LANG},
        {\it Real and Functional Analysis, vol. 142 of Graduate Texts in Mathematics}, Springer-Verlag, New York, 1993       

 \bibitem{lazlea}
        {\sc  A.C. LAZER, D.F. LEACH},
        {\it Bounded perturbations of forced harmonic oscillators at resonance}, Ann. Mat. Pure Appl., \textbf{82} (1969) 49--68.   
               
\bibitem{marrod2}
        {\sc D. MARONCELLI, J. RODR\'{I}GUEZ}, 
        {\it Existence theory for nonlinear Sturm-Liouville problems with non-local boundary conditions}, Differential Equations and Applications {\bf 10}, 2 (2018), 147--161.            
        
\bibitem{marrod}
        {\sc D. MARONCELLI, J. RODR\'{I}GUEZ}, 
        {\it Periodic behaviour of nonlinear, second-order discrete dynamical systems}, Journal of Difference Equations and Applications {\bf 22}, 2 (2016), 280--294.      
  
 \bibitem{gal}
        {\sc  J. RODR\'{I}GUEZ},
        {\it Galerkin's method for ordinary differential equations subject to generalized nonlinear boundary conditions}, J. Differential Equations, \textbf{97}, 1 (1992) 112--136.        
  
\bibitem{sl05}
        {\sc  J. RODR\'{I}GUEZ},
        {\it Nonlinear discrete Sturm-Liouville problems}, J. Math Anal. Appl., \textbf{308}, 1 (2005) 380--391.        
        
\bibitem{ja1}
        {\sc J. RODR\'{I}GUEZ, Z. ABERNATHY},
        {\it On the Solvability of Nonlinear Sturm-Liouville Problems}, Journal of Mathematical Analysis and Applications, {\bf 387}, 1 (2012), 310-319.          
        
\bibitem{jsuar1}
        {\sc J. RODR\'{I}GUEZ, A.J. SUAREZ}, 
        {\it On nonlinear perturbations of Sturm-Liouville problems in discrete and continuous settings}, Differential Equations and Applications {\bf 8}, 3  (2016), 319--334.
        
\bibitem{jsuar2}
        {\sc J. RODR\'{I}GUEZ, A.J. SUAREZ}, 
        {\it Existence of solutions to nonlinear boundary value problems}, Differential Equations and Applications {\bf 9}, 1  (2017), 1--11.        
        
\bibitem{sweet}
        {\sc J. RODR\'{I}GUEZ, D. SWEET}, 
        {\it Projection methods for nonlinear boundary value problems}, Journal of Differential Equations {\bf 58}, (1985), 282--293.
        
\bibitem{pad}
        {\sc J. RODR\'{I}GUEZ, P. TAYLOR}, 
        {\it Scalar discrete nonlinear multipoint boundary value problems}, Journal of Mathematical Analysis and Applications {\bf 330}, 2 (2007), 876--890.
        
 \bibitem{jfrod}
        {\sc  J. F. RODR\'{I}GUEZ},
        {\it Existence theory for nonlinear eigenvalue problems}, Appl. Anal., \textbf{87} (2008) 293--301.
       
\bibitem{ura}
        {\sc  M. URABE},
        {\it Galerkin's procedure for nonlinear periodic systems}, Archive for Rational Mechanics and Analysis, \textbf{20}, (1965) 120--152.             
        
         







\end{thebibliography}
\end{document}